\documentclass{amsart}

\usepackage[english]{babel}
\usepackage{hyperref}
\usepackage{graphicx}

\usepackage{amsmath}
\usepackage{amssymb}

\usepackage[latin1]{inputenc}
\usepackage{amsthm}
\usepackage{amsfonts}
\usepackage{enumitem}
\usepackage{xcolor}

\newtheorem*{theorem*}{Theorem}
\newtheorem*{definition*}{Definition}
\newtheorem{theorem} {Theorem}[section]
\newtheorem{corollary}[theorem]{Corollary}
\newtheorem{lemma}[theorem]{Lemma}
\newtheorem{proposition}[theorem]{Proposition}

\newtheorem{example}[theorem]{Example}
\newtheorem{remark}[theorem]{Remark}

\newtheorem{definition}[theorem]{Definition}

\newcommand{\aff} {\operatorname{aff} }
\newcommand{\conv} {\operatorname{conv} }

\newcommand{\vol} {\operatorname{vol} }
\newcommand{\wert} {\operatorname{vert}}
\newcommand{\Z}{\mathbb{Z}}
\newcommand{\R}{\mathbb{R}}
\newcommand{\N}{\mathbb{N}}
\newcommand{\gen}[1]{\langle #1\rangle_{\Z}}

\begin{document}

\title{Non-spanning lattice $3$-polytopes}
\author{M\'onica Blanco}
\author{Francisco Santos}

\address
[M.~Blanco]
{
Mathematical Sciences Research Institute,
17 Gauss Way,
CA 94720 Berkeley, U. S. A.}
\email{m.blanco.math@gmail.com}

\address
[F.~Santos ]
{
Departamento de Matem\'aticas, Estad\'istica y Computaci\'on,
Universidad de Cantabria,
Av. de Los Castros 48,
39005 Santander, Spain
}
\email{francisco.santos@unican.es}

\thanks{Partially supported by grants MTM2014-54207-P and MTM2017-83750-P (both authors) and BES-2012-058920 (M.~Blanco) of the Spanish Ministry of Economy and Competitiveness, and by the Einstein Foundation Berlin (F.~Santos).
Both authors were also supported by the National Science Foundation under Grant No. DMS-1440140 while they were in residence at the Mathematical Sciences Research Institute in Berkeley, California, during the Fall 2017 semester.}

\begin{abstract}
We completely classify non-spanning $3$-polytopes, by which we mean lattice $3$-polytopes whose lattice points do not affinely span the lattice.
We show that, except for six small polytopes (all having between five and eight lattice points), every non-spanning $3$-polytope $P$ has the following simple description: $P\cap \Z^3$ consists of either (1) two lattice segments lying in parallel and consecutive lattice planes or (2) a lattice segment together with three or four extra lattice points placed in a very specific manner.

From this description we conclude that all the empty tetrahedra in a non-spanning $3$-polytope $P$ have the same volume and they form a triangulation of $P$, and we compute the $h^*$-vectors of all non-spanning $3$-polytopes.

We also show that all spanning $3$-polytopes contain a unimodular tetrahedron, except for two particular $3$-polytopes with five lattice points.
\end{abstract}

\maketitle

\setcounter{tocdepth}{1}
\tableofcontents


\section{Introduction and statement of results}

A lattice $d$-polytope is a polytope $P\subset\R^d$ with vertices in $\Z^d$ and with $\aff(P)=\R^d$. We call \emph{size} of $P$ its number of lattice points and \emph{width} the minimum length of the image $f(P)$ when $f$ ranges over all affine non-constant functionals $f:\R^d\to \R$ with $f(\Z^d)\subseteq\Z$.
That is, the minimum lattice distance between parallel lattice hyperplanes that enclose $P$.
\smallskip

In our papers~\cite{5points,6points,quasiminimals} we have enumerated all lattice $3$-polytopes of size 11 or less and of width greater than one. This classification makes sense thanks to the following result~\cite[Theorem 3]{5points}: \emph{for each $n\in \N$ there are only finitely many lattice $3$-polytopes of width greater than one and with exactly $n$ lattice points}. 
Here and in the rest of the paper we consider lattice polytopes modulo \emph{unimodular equivalence} or \emph{lattice isomorphism}. That is, we consider $P$ and $Q$ \emph{isomorphic} (and write $P\cong Q$) if there is an affine map $f:\R^d\to \R^d$ with $f(\Z^d)=\Z^d$ and $f(P)=Q$.\medskip

As a by-product of the classification we noticed that most lattice $3$-polytopes are ``lattice-spanning'', according to the following definition:

\begin{definition*}
Let $P\subset \R^d$ be a lattice $d$-polytope. We call \emph{sublattice index} of $P$ the index, as a sublattice of $\Z^d$, of the affine lattice generated by $P\cap \Z^d$.
$P$ is called \emph{lattice-spanning} if it has index $1$. 
We abbreviate sublattice index and lattice-spanning as \emph{index} and \emph{spanning}. 
\end{definition*}

In this paper we completely classify non-spanning lattice $3$-polytopes. 
Part of our motivation comes from the recent results of Hofscheier et al.~\cite{HKN-a,HKN-b} on $h^*$-vectors of spanning polytopes (see Theorem~\ref{thm:HKN}). In particular, in Section~\ref{sec:hstar} we compute the $h^*$-vectors of all non-spanning $3$-polytopes and show that they still satisfy the inequalities proved by Hofscheier et al.~for spanning polytopes, with the exception of empty tetrahedra that satisfy them only partially.\medskip

In dimensions $1$ and $2$, every lattice polytope contains a unimodular simplex, i.~e., a lattice basis, and is hence lattice-spanning.
In dimension $3$ it is easy to construct infinitely many lattice polytopes of width $1$ and of any index $q\in \N$, generalizing White's empty tetrahedra (\cite{White}). Indeed, for any positive integers $p,q,a,b$ with $\gcd(p,q)=1$ the lattice tetrahedron 
\[
T_{p,q}(a,b):=\conv\{(0,0,0),(a,0,0),(0,0,1),(bp,bq,1)\}
\]
has index $q$, width $1$, size $a+b+2$ and volume $abq$ (see a depiction of it in Figure~\ref{fig:width1_nonspanning}).
\begin{figure}[h]
\centerline{\includegraphics[scale=.6]{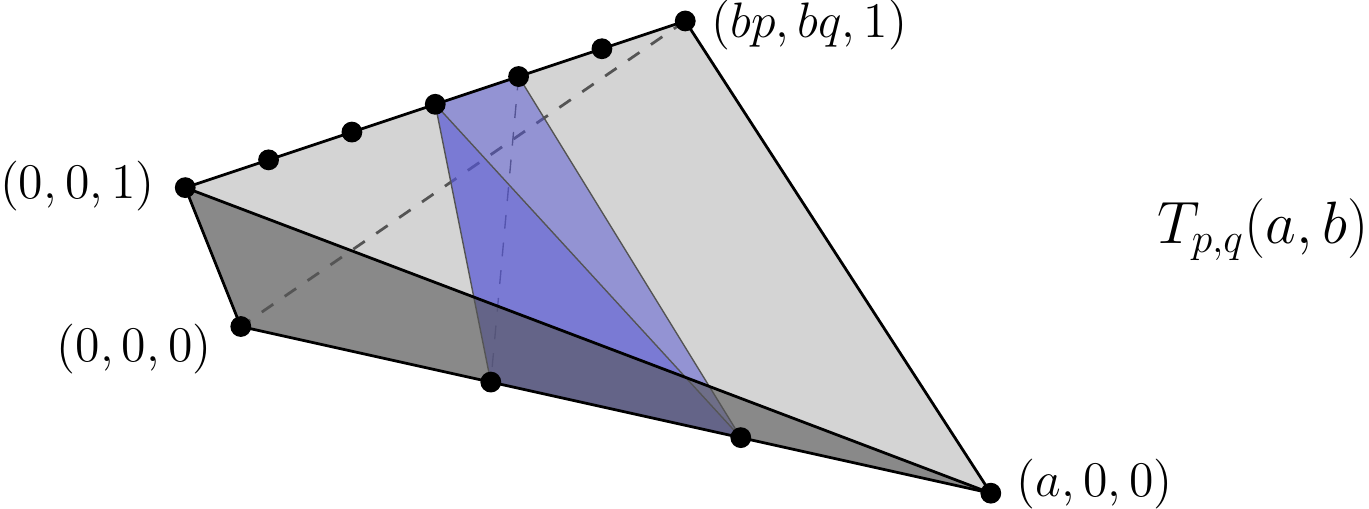}}
\caption{A polytope $T_{p,q}(a,b)$. An empty tetrahedron in it is highlighted.}
\label{fig:width1_nonspanning}
\end{figure}
Here and in the rest of the paper we consider the volume of lattice polytopes normalized to the lattice, so that it is always an integer and the normalized volume of a simplex $\conv(v_0, \dots, v_d)$ equals its determinant $\left\vert\det\begin{pmatrix}v_0 \cdots v_d \\ 1 \cdots 1\end{pmatrix}\right\vert$.

\begin{lemma}[Corollary~\ref{cor:width1-iff}]
\label{lemma:width1_intro}
Every non-spanning $3$-polytope of width one is isomorphic to some $T_{p,q}(a,b)$.
\end{lemma}
\smallskip

For larger width, the complete enumeration of lattice $3$-polytopes of width larger than one up to size $11$ shows that their index is always in $\{1,2,3,5\}$. The numbers of them for each index and size are as given in Table~\ref{table:numbers} (copied from Table~6 in~\cite{quasiminimals}).
\begin{table}[htb]
\begin{center}
\begin{tabular}{l|rrrrrrr}
{\bf size}      &  {\bf 5}   &  {\bf  6}  &    {\bf 7}  &   {\bf 8}  &   {\bf 9}  &   {\bf 10}  &   {\bf 11}\\
\hline
 {\bf index 1} & 7   &  71   &  486  & 2658  &   11680  &   45012 & 156436\rule{0pt}{3ex} \\
 {\bf index 2} & 0  &   2  &    8  &   14  &   15  &   19  &   24\\
 {\bf index 3} & 1   &   3  &    2  &    3   &   3   &   4    &  4\\
 {\bf index 5} & 1   &   0  &    0  &    0   &   0   &   0    &  0\\
\end{tabular}
\end{center}
\smallskip
\caption{Lattice $3$-polytopes of width $>1$ and size up to $11$, classified according to sublattice index. Table taken from~\cite{quasiminimals}.}
\label{table:numbers}
\end{table}
This data seems to indicate that apart from a few small exceptions there are about $\Theta(n)$ polytopes of index three, about $\Theta(n^2)$ of index two, and none of larger indices. Infinite families of lattice $3$-polytopes of indices two and three that match these asymptotics are given in the following statement:

\begin{lemma}[See Section~\ref{sec:families}]
\label{lemma:inf_families}
Let $S_{a,b}:=\conv\{(0,0,a),(0,0,b)\}$ be a lattice segment, with $a,b\in\Z$, $a\le 0 < b$.
Then the following are non-spanning lattice $3$-polytopes of width $>1$ (see pictures of them in Figure~\ref{fig:projs_families}):
\begin{itemize}
\item $\widetilde F_1(a,b):=\conv\left( S_{a,b}\cup \{(-1,-1,0), (2,-1,1), (-1,2,-1)\}\right)$, of index $3$.\smallskip
\item $\widetilde F_2(a,b):=\conv\left( S_{a,b}\cup \{(-1,-1,1), (1,-1,0), (-1,1,0)\}\right)$, with $(a,b)\neq(0,1)$, of index $2$. \smallskip
\item $\widetilde F_3(a,b,k):=\conv\left( S_{a,b}\cup \{(-1,-1,1), (1,-1,0), (-1,1,0),(1,1,2k-1)\}\right)$, for any $k\in \{0,\dots, b\}$ and $(a,b,k)\neq(0,1,1)$, of index $2$. \smallskip
\item $\widetilde F_4(a,b):=\conv\left( S_{a,b}\cup \{(-1,-1,1), (1,-1,0), (-1,1,0),(3,-1,-1)\}\right)$, of index 2.
\end{itemize}
\end{lemma}

\begin{figure}[htb]
\centerline{
\quad $\widetilde F_1(a,b)$ \qquad \qquad \quad
\quad $\widetilde F_2(a,b)$ \qquad \quad
\quad $\widetilde F_3(a,b,k)$ \qquad \qquad
\quad $\widetilde F_4(a,b)$ 
\quad
}
\vspace{3ex}
\centerline{
\includegraphics[scale=.45]{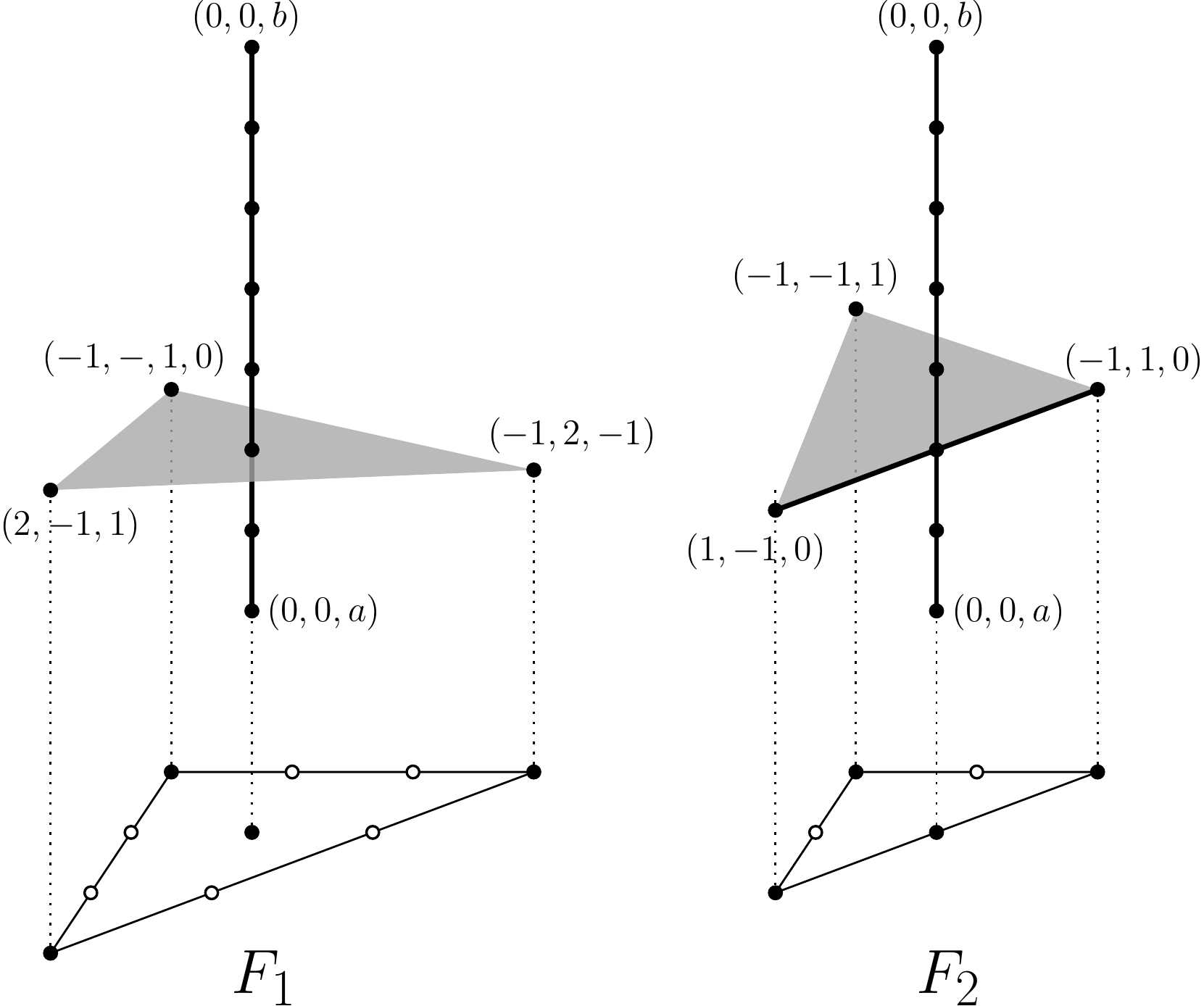}
\includegraphics[scale=.45]{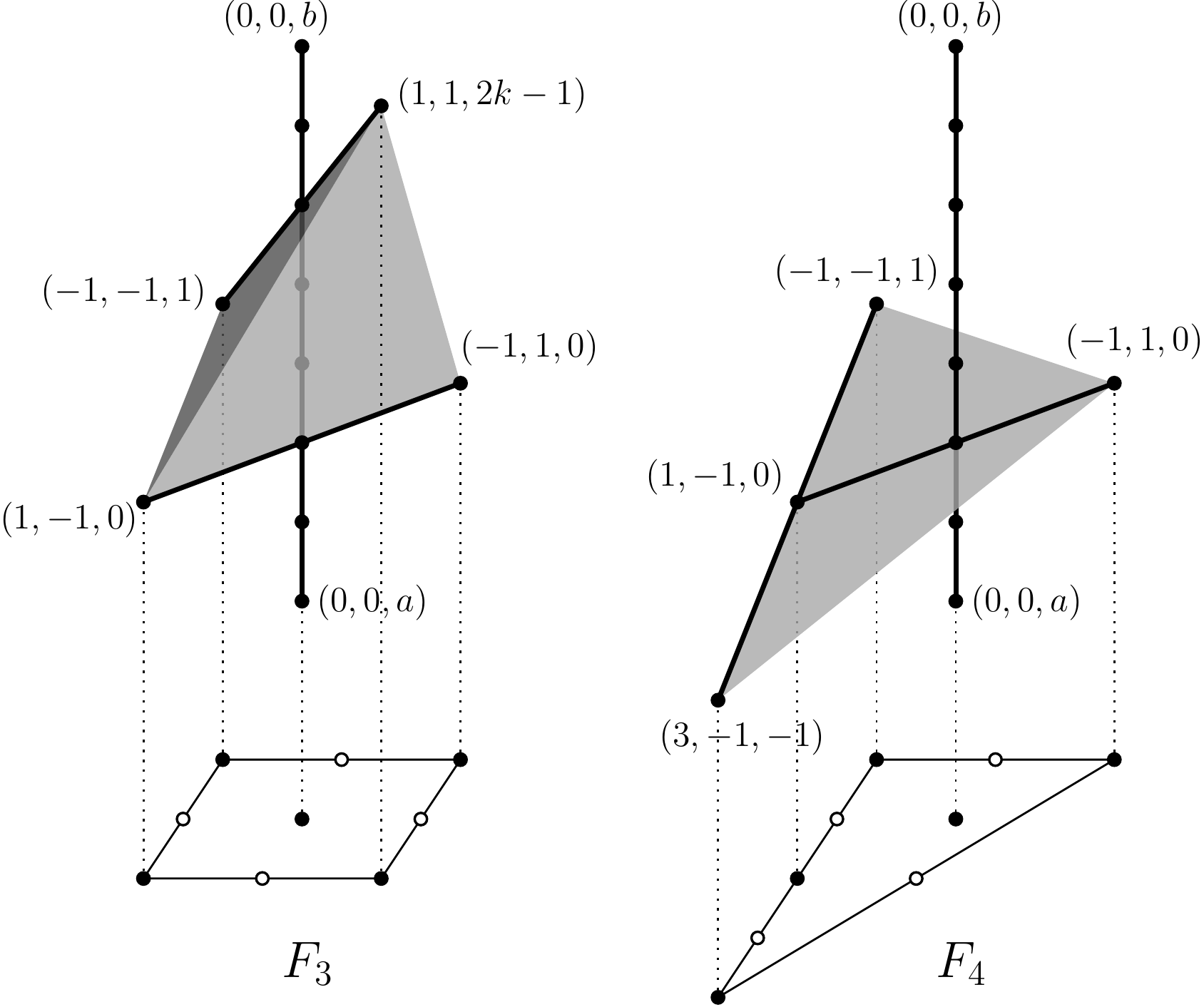}
}
\caption{The non-spanning $3$-polytopes of Lemma~\ref{lemma:inf_families}. Black dots in the top figures are the lattice points in $\widetilde F_i(a,b,\_)$. Black segments represent segments of at least three collinear lattice points, and the gray triangles or tetrahedron are the convex hulls of the lattice points not in the spike. 
The lattice point configurations $F_i$ arise as the projections of the configurations $\widetilde F_i(a,b,\_) \cap \Z^3$ in the direction of the spike.
Black dots are points of $F_i$ (projection of lattice points of $\widetilde F_i(a,b,\_)$), and white dots are lattice points in $\conv(F_i)\setminus F_i$.
}
\label{fig:projs_families}
\end{figure}

Observe that in all cases of Lemma~\ref{lemma:inf_families} the lattice points in the polytope are those in the segment $S_{a,b}$ plus three or four additional points. We call $S_{a,b}$ (or, sometimes, the line containing it) the \emph{spike} of $\widetilde F_i(a,b,\_)$, because it makes these polytopes be very closely related to (some of) the \emph{spiked polytopes} introduced in~\cite{quasiminimals} (see the proof of Theorem~\ref{thm:spiked-index} for some properties of them). 

For all polytopes in the same family $\widetilde F_i$ the projection of the lattice points along the direction of the spike is the same lattice point configuration, the $F_i$ depicted in the bottom of Figure~\ref{fig:projs_families}. The spike projects to the origin, and the other three or four points in each $\widetilde F_i(a,b,\_)$ project bijectively to three or four points in $F_i$.

The number of polytopes in the families of Lemma~\ref{lemma:inf_families}, taking redundancies into account, is computed in Corollary~\ref{coro:numbers}.
Table~\ref{table:numbers2} shows the result for sizes $5$ to $11$.
\begin{table}[htb]
\begin{center}
\begin{tabular}{l|rrrrrrr}
{\bf size}      &  {\bf 5}   &  {\bf  6}  &    {\bf 7}  &   {\bf 8}  &   {\bf 9}  &   {\bf 10}  &   {\bf 11}\\
\hline
 {\bf index 2} & 0  &   2  &    7  &   11  &   15  &   19  &   24\rule{0pt}{3ex}\\
 {\bf index 3} & 1   &   2  &    2  &    3   &   3   &   4    &  4\\
\end{tabular}
\end{center}
\smallskip
\caption{The number of polytopes of size $\le 11$ in the families of Lemma~\ref{lemma:inf_families}}
\label{table:numbers2}
\end{table}

Comparing those numbers to Table~\ref{table:numbers} we see that the only discrepancies are six missing polytopes in sizes five to eight. 
Our main result in this paper is that indeed, Lemma~\ref{lemma:inf_families} together with six exceptions gives the full list of non-spanning lattice $3$-polytopes of width larger than one.

\begin{theorem}[Classification of non-spanning $3$-polytopes]
\label{thm:main_exceptions}
Every non-spanning lattice $3$-polytope either has width one (hence it is isomorphic to a $T_{p,q}(a,b)$ by Lemma~\ref{lemma:width1_intro}) or is isomorphic to one in the infinite families of Lemma~\ref{lemma:inf_families}, or is isomorphic to one of the following six polytopes, depicted in Figure~\ref{fig:exceptions_intro}:
\begin{itemize}
\item Two tetrahedra of sizes $5$ and $6$, with indices $5$ and $3$ respectively:
\[
E_{(5,5)}:=\conv\{
(0, -2, 1), (1, 0, -1), (1, 1, 1), (-2, 1, -1)\}.
\]
\[
E_{(6,3)}:=\conv\{
(1,0,0), (-1,-1,0), (1,2,3),(-1,1,-3)\}.
\]

\item A pyramid and two double pyramids over the following triangle of size $6$:
\[
B:=\conv\left(\{(-1,-1,0),(2,0,0),(1,2,0)\}\right).
\]
Namely, the following polytopes of sizes 7 and 8, all of index $2$:
\[
E_{(7,2)}:=\conv(B\cup\{(0,-1,2)\})
\]
\[
E_{(8,2)}^1:=\conv(B\cup\{(0,-1,2),(0,1,-2)\}),
\]
\[ 
E_{(8,2)}^2:=\conv(B\cup\{(0,-1,2),(-2,-1,-2)\}).
\]
 
\item A tetrahedron of size $8$ and index $2$:
\[
E_{(8,2)}^3:=\conv\{(0,-1,-1), (2,0,2), (1,2,-1),(-1,1,2)\}.
\] 
\end{itemize}
\end{theorem}

\begin{figure}[h]
\centerline{
\begin{tabular}{c|c}
\hline
\includegraphics[scale=.6]{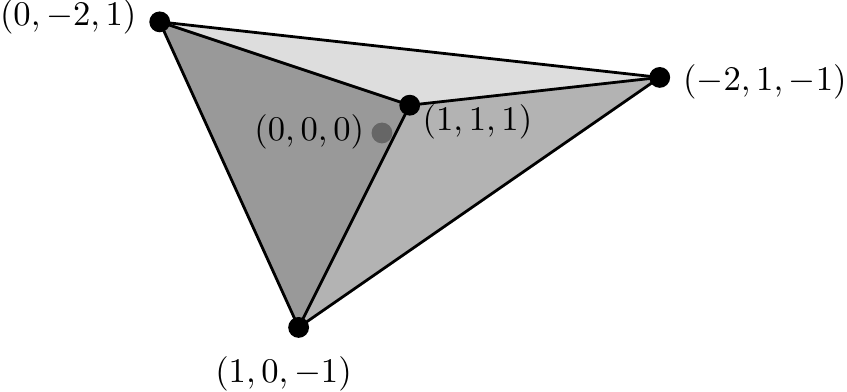}
&\includegraphics[scale=.6]{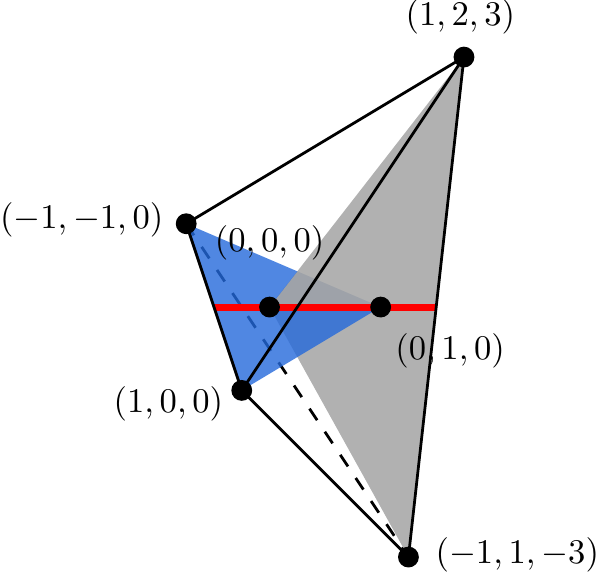}\rule{0pt}{24ex}
\\ 
\large$E_{(5,5)}$ & $E_{(6,3)}$\rule[-2ex]{0pt}{6ex}\\
\end{tabular}
}\hrule

\centerline{
\begin{tabular}{c|c|c}
\includegraphics[scale=.6]{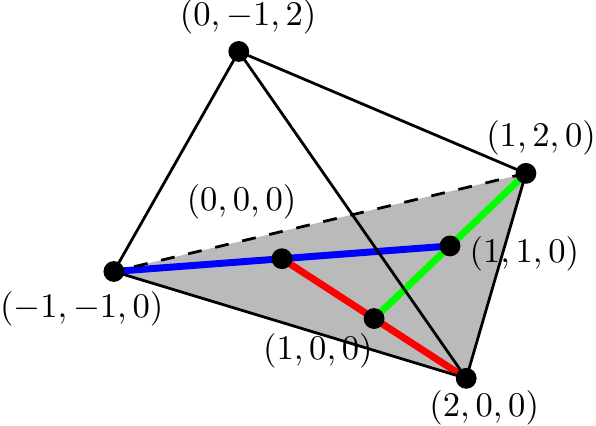}\quad
&
\includegraphics[scale=.6]{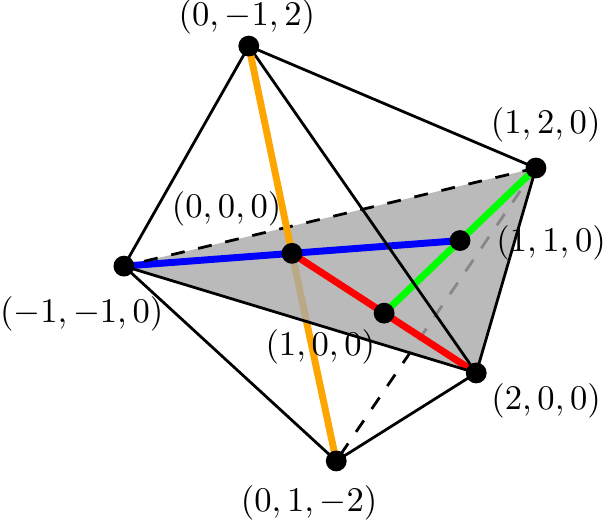}\quad
&
\includegraphics[scale=.6]{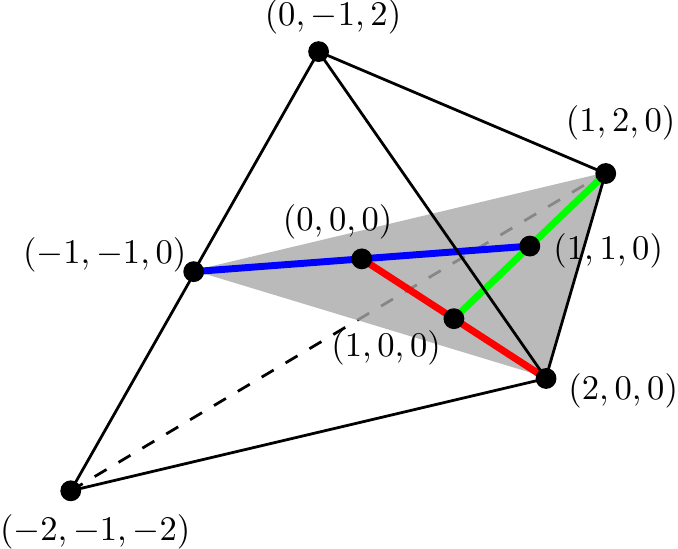}\rule{0pt}{24ex}
\\
\large$E_{(7,2)}$ & $E_{(8,2)}^1$ & $E_{(8,2)}^2$ \rule[-2ex]{0pt}{6ex}\\
\end{tabular}
}\hrule

\centerline{
\includegraphics[scale=.6]{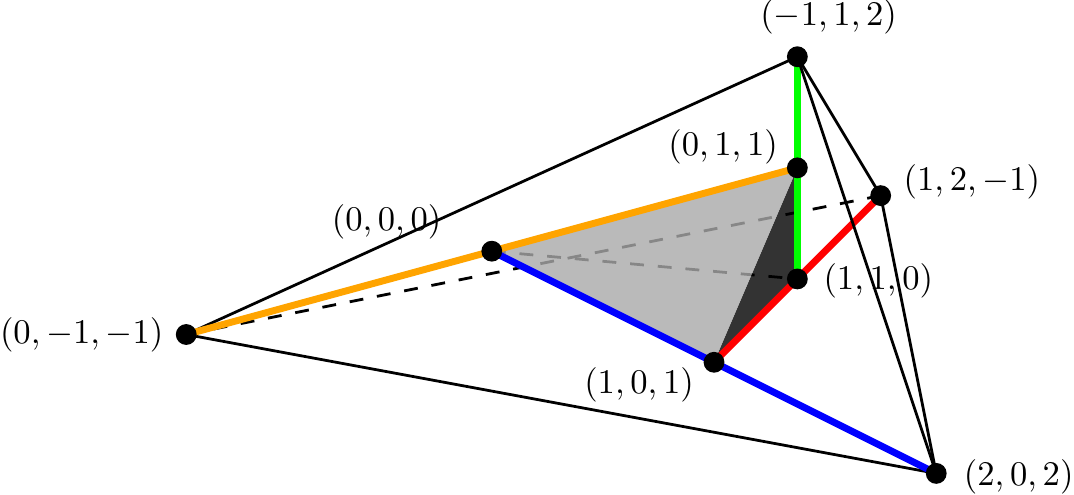}\rule{0pt}{23ex}
}
\centerline{\large $E_{(8,2)}^3$\rule[-2ex]{0pt}{6ex}}
\hrule
\normalsize

\caption{The six non-spanning $3$-polytopes of Theorem~\ref{thm:main_exceptions}.}
\label{fig:exceptions_intro}
\end{figure}

\begin{corollary}[Corollary~\ref{coro:numbers}]
\label{coro:main}
For any $n\ge 9$ all non-spanning $3$-polytopes of size $n$ and width larger than one belong to the families of Lemma~\ref{lemma:inf_families}. There are $\left\lfloor {(n-3)(n+1)}/{4}\right\rfloor$ of index two, $\lceil \frac{n-3}{2}\rceil$ of index three, and none of larger index.
\qed
\end{corollary}

\smallskip

As a consequence of our classification for non-spanning $3$-polytopes we have the following statement. Recall that an \emph{empty} tetrahedron is a lattice tetrahedron whose only lattice points are its vertices.

\begin{corollary}
\label{coro:volume_i}
Let $P$ be a non-spanning lattice $3$-polytope. Then the collection of all empty tetrahedra in $P$ is a triangulation, and all such tetrahedra have volume equal to the index of $P$.
\end{corollary}

Observe that the triangulation in the statement is necessarily the unique triangulation of $P$ with vertex set equal to $P\cap \Z^3$.

\begin{proof}
\begin{itemize}

\item For $T_{p,q}(a,b)$, since all lattice points lie on two lines, every empty tetrahedron has as vertices two consecutive lattice points on each of the lines. These tetrahedra have volume equal to the index (in this case $q$) and, together, form a triangulation: the join of the two paths along the lines.

\item Every polytope in the families $\widetilde F_i$ has the property that the triangle formed by any three lattice points outside the spike meets the spike at a lattice point. Thus, every empty tetrahedron has two vertices along the spike and two outside the spike. The collection of all of such tetrahedra forms a triangulation: the join of the path along the spike and a path (for the polytopes in the family $\widetilde F_2$) or a cycle (for the rest) outside the spike. Also, these tetrahedra have the volume of a triangle with vertices in the projection $F_i$ that has the origin as a vertex. Any such triangle in $F_i$ has the same volume as the index of the polytopes in the family $\widetilde F_i$.
\item For the six exceptions of Theorem~\ref{thm:main_exceptions}, the statement can be checked with geometric arguments or with computer help. Details are left to the reader.
\end{itemize}
\end{proof}

Corollary~\ref{coro:volume_i} is not true in dimension $4$, as the following example shows:
\begin{example}
\rm
Let
\[P=\conv\left\{
\begin{array}{rrrr}
(\,\;\; 1,& 0,& 0,& 0 ),\\
(\,\;\; 0,& 1,& 0,& 0 ),\\
(\,\;\; 0,& 0,& 1,& 0 ),\\
(-2,&-1,&-1,& 0 ),\\
(\,\;\; 1,& 1,& 1,& 2 )\,\,
\end{array}
\right\}
\]
This is a lattice $4$-simplex with six lattice points: its five vertices plus the origin, which lies in the relative interior of the facet given by the first four vertices. Thus, the volumes of the empty tetrahedra in $P$ are the absolute values of the $4\times 4$ minors in the matrix (excluding the zero minor). These values are four, two, two and two, and $P$ has index two.
We thank Gabriele Balletti for providing (a variation of) this counterexample.
\end{example}

In the same vein, one can ask whether all spanning $3$-polytopes have a unimodular tetrahedron. The answer is that only two do not.

\begin{theorem}[See Section~\ref{sec:spanning}]
\label{thm:no_uni_tetra}
The only lattice-spanning $3$-polytopes not containing a unimodular tetrahedron are the following two tetrahedra of size five:
\begin{itemize}
\item $E_{(5,1)}^1:=\conv\{
(1,0,0), (0,0,1),(2,7,1),(-1,-2,-1)\}$ with four empty tetrahedra of volumes $2$, $3$, $5$ and $7$.

\item $E_{(5,1)}^2:=\conv\{
(1,0,0), (0,0,1),(3,7,1),(-2,-3,-1)\}$ with four empty tetrahedra of volumes $3$, $4$, $5$ and $7$.

\end{itemize}
\end{theorem}

The structure of the paper is as follows. After a brief introduction and remarks about the sublattice index in Section~\ref{sec:prelim}, Section~\ref{sec:width_one} proves the (easy) classification of non-spanning $3$-polytopes of width one (Lemma~\ref{lemma:width1_intro}), and Section~\ref{sec:families} is devoted to the study of the infinite families of non-spanning $3$-polytopes of Lemma~\ref{lemma:inf_families}.

Section~\ref{sec:main} proves our main result, Theorem~\ref{thm:main_exceptions}. 
This is the most complicated part of the paper, relying substantially in our results from~\cite{quasiminimals}. A sketch of the proof is as follows:
For small polytopes, Theorem~\ref{thm:main_exceptions} follows from comparing Tables~\ref{table:numbers} and~\ref{table:families} (together with the easy observation that the six polytopes described in Theorem~\ref{thm:main_exceptions} are not isomorphic to one another or to the ones in Lemma~\ref{lemma:inf_families}).
For polytopes of larger size we use induction on the size (taking the enumeration of size $\le 11$ as the base case) and we prove the following:

\begin{itemize}

\item In Section~\ref{sec:merging} we look at polytopes that cannot be obtained \emph{merging} two smaller ones, in the sense of~\cite{quasiminimals}. These polytopes admit quite explicit descriptions that allow us to prove that the only non-spanning ones (for sizes $\ge 8$) are those of the form $\widetilde F_i(0,b)$ for $i\in \{1,2\}$ (Theorem~\ref{thm:spiked-index}). Since merging can only decrease the index, this immediately implies that all non-spanning $3$-polytopes of width $>1$ and size $\ge 8$ have index at most three (Corollary~\ref{coro:width>1-index}).

\item In Sections~\ref{sec:index3} and~\ref{sec:index2} we look at \emph{merged} polytopes of indices three and two, respectively, and prove that, with the five exceptions mentioned in Theorem~\ref{thm:main_exceptions}, they all belong to the families of Lemma~\ref{lemma:inf_families}. 
\end{itemize}

Section~\ref{sec:spanning} proves Theorem~\ref{thm:no_uni_tetra}.
\medskip


\section{Sublattice index}
\label{sec:prelim}

By a \emph{lattice point configuration} we mean a finite subset $A\subset \Z^d$ that affinely spans $\R^d$. We denote by $\gen{A}$ the affine lattice generated by $A$ over the integers:

\[
\gen{A}:=\left\{\sum_i \lambda_ia_i \; \vert \; a_i\in A,\lambda_i\in\Z, \sum_{i}\lambda_i=1\right\}
\]

Since $A$ is affinely spanning, $\gen{A}$ has finite index as a sublattice of $\Z^d$.

\begin{definition}
The \emph{sublattice index} of $A$ is the index of $\gen{A}$ in $\Z^d$.
We say that $A$ is \emph{lattice-spanning} if its sublattice index equals $1$. That is, if $\gen{A}=\Z^d$. 
\end{definition}

\begin{remark}
If $A=P\cap \Z^d$ for $P$ a lattice $d$-polytope, we call \emph{sublattice index of $P$} the sublattice index of $A$, and say that $P$ is \emph{lattice-spanning} if $A$ is.
\end{remark}

\begin{lemma}\label{lemma:index} 
 Let $A\subset \Z^d$ be a lattice point configuration. 
\begin{enumerate}
\item The sublattice index of $A$ divides the sublattice index of every subconfiguration $B$ of $A$.
\label{item:subpolytope}
\item Let $\pi:\Z^d\twoheadrightarrow\Z^s$, for $s <d$, be a lattice projection. Then the sublattice index of $\pi(A)$ divides the sublattice index of $A$.
\label{item:projection}
\end{enumerate}
\end{lemma}
\begin{proof}
In part (1), the injective homomorphism $ \gen{B} \to \gen{A}$ induces a surjective homomorphism $\Z^d / \gen{B} \to\Z^d / \gen{A}$. 
In part (2), the lattice projection $\pi$ induces a surjective homomorphism $\Z^d / \gen{A} \to\Z^s / \gen{\pi(A)}$.
\end{proof}

It is very easy to relate the index of a lattice polytope or point configuration to the volumes of (empty) simplices in it. 

\begin{lemma}
\label{lemma:index_gcd}
Let $A$ be a lattice point configuration of dimension $d$. Then, the sublattice index of $A$ equals the gcd of the volumes of all the lattice $d$-simplices with vertices in $A$. 
\end{lemma}

\begin{remark}
Observe that the $\gcd$ of volumes of \emph{all} simplices in $A$ equals the $\gcd$ of volumes of simplices \emph{empty in $A$}, by which we mean simplices $T$ such that $T\cap A= \wert(T)$. Indeed, if $T$ is a non-empty simplex, then $T$ can be triangulated into empty simplices $T_1,\dots,T_k$. Since $\vol(T)= \sum \vol(T_i)$, we have that $\gcd(\vol(T_1),\dots, \vol(T_n) )=\gcd(\vol(T_1),\dots, \vol(T_n) , \vol(T))$.
\end{remark}

\begin{proof}
Without loss of generality assume that the origin is in $A$. Then, the sublattice index of $A$ equals the gcd of all maximal minors of the $d\times |A|$ matrix having the points of $A$ as columns. These minors are the (normalized) volumes of lattice $d$-simplices with vertices in $A$.
\end{proof}


\section{Non-spanning $3$-polytopes of width one}
\label{sec:width_one}

\begin{lemma}
\label{lemma:width1-index}
Let $P$ be a lattice $3$-polytope of width one. Then $P$ either contains a unimodular tetrahedron or it equals the convex hull of two lattice segments lying in consecutive parallel lattice planes.
\end{lemma}

\begin{proof}
Since $P$ has width one, its lattice points are distributed in two consecutive parallel lattice planes. If $P$ has three non collinear lattice points in one of these planes we can assume without loss of generality that they form a unimodular triangle. Then, these three points together with any point in the other plane (there exists at least one) form a unimodular tetrahedron.

If $P$ does not have three non-collinear points in one of the two planes then all the points in each of the planes are contained in a lattice segment.
\end{proof}

\begin{proposition}
\label{pro:segments_case}
The convex hull of two lattice segments lying in consecutive parallel lattice planes is equivalent to 
\[
T_{p,q}(a,b):=\conv\{(0,0,0),(a,0,0),(0,0,1),(bp,bq,1)\}
\]
for some $0\le p < q$ with $\gcd(p,q)=1$, and $a,b\ge 1$.

The sublattice index of $T_{p,q}(a,b)$ is $q$ and its size is $a+b+2$. 
\end{proposition}

See Figure~\ref{fig:width1_nonspanning} for an illustration of $T_{p,q}(a,b)$.

\begin{proof} 
Let $P$ be the convex hull of two lattice segments lying in consecutive parallel lattice planes.
Without loss of generality we assume that the segments are contained in the planes $\{z=0\}$ and $\{z=1\}$, respectively.
Let $a$ and $b$ be the lattice length of the two segments, so that the size of $P$ is indeed $a+b+2$. 

By a unimodular transformation there is no loss of generality in assuming the first segment to be $\conv\{(0,0,0),(a,0,0)\}$ and the second one to contain $(0,0,1)$ as one of its endpoints. Then the second endpoint is automatically of the form $(bp,bq,1)$ for some coprime $p,q\in \Z$ and with $q\ne 0$ in order for $P$ to be full-dimensional. We can assume $q>0$ since the case $q<0$ is symmetric.
Also, since the unimodular transformation $(x,y,z)\mapsto (x\pm y,y,z)$ fixes $(0,0,0)$, $(a,0,0)$ and $(0,0,1)$ and sends $(p,q,1)\mapsto (p\pm q,q,1)$, there is no loss of generality in assuming $0\le p <q$.

Since all lattice points of $P$ lie in the lattice $y\equiv0\pmod q$ its index is a multiple of $q$. Since $P$ contains the tetrahedron $\conv\{(0,0,0)(1,0,0), (0,0,1),(p, q,1)\}$, of determinant $q$, its index is exactly $q$.
\end{proof}

\begin{corollary}
\label{cor:width1-iff}
Let $P$ be a lattice $3$-polytope of width one. Then
\begin{itemize}
\item $P$ is lattice-spanning if, and only if, it contains a unimodular tetrahedron.
\item $P$ has index $q>1$ if, and only if, $P \cong T_{p,q}(a,b)$ for some $0\le p < q$ with $\gcd(p,q)=1$ and $a,b\ge 1$.
\end{itemize}
\qed
\end{corollary}


\section{Four infinite non-spanning families}
\label{sec:families}

In this section we study the infinite families of non-spanning $3$-polytopes introduced in Lemma~\ref{lemma:inf_families}, whose definition we now recall:
\begin{align*}
\widetilde F_1(a,b)&:=\ \conv\big( S_{a,b}\cup \{(-1,-1,0), (2,-1,1), (-1,2,-1)\}\big),\\
\widetilde F_2(a,b)&:=\ \conv\big( S_{a,b}\cup \{(-1,-1,1), (1,-1,0), (-1,1,0)\}\big),\\
\widetilde F_3(a,b,k)&:=\ \conv\big( S_{a,b}\cup \{(-1,-1,1), (1,-1,0), (-1,1,0),(1,1,2k-1)\}\big),\\
\widetilde F_4(a,b)&:=\ \conv\big( S_{a,b}\cup \{(-1,-1,1), (1,-1,0), (-1,1,0),(3,-1,-1)\}\big).
\end{align*}
In all cases, $S_{a,b}:=\conv\{(0,0,a),(0,0,b)\}$ with $a,b\in\Z$ and $a\le 0 < b$, and in the third family, $k\in \{0,\dots, b\}$.

The statement of Lemma~\ref{lemma:inf_families} is that all these polytopes are non-spanning and have width $>1$, with the exceptions of $\widetilde F_2(0,1)$ and $\widetilde F_3(0,1,1)$, which clearly have width one with respect to the third coordinate.
For the proof, recall that we call the lattice segment $S_{a,b}$ the \emph{spike}, and that apart of the $b-a+1$ lattice points in the spike these polytopes only have three (in the families $\widetilde F_1$ and $\widetilde F_2$) or four (in the families $\widetilde F_3$ or $\widetilde F_4$) other lattice points.

\begin{proof}[Proof of Lemma~\ref{lemma:inf_families}]
The lattice points in $\widetilde F_1(a,b)$ generate the sublattice $\{(x,y,z) \in \Z^3 : x-y \equiv 0 \pmod 3\}$, of index $3$. Those in $\widetilde F_2(a,b)$, $\widetilde F_3(a,b,k)$ and $\widetilde F_4(a,b)$ generate the sublattice $\{(x,y,z) \in \Z^3 : x+y \equiv 0 \pmod 2\}$, of index $2$.

For the width we consider two cases: Functionals constant on $S_{a,b}$ (that is, not depending on the third coordinate) project to functionals in $\Z^2$ on one of the configurations $F_i$ of Figure~\ref{fig:projs_families}, of which $F_1$ has width $3$ and the others width $2$.

Functionals that are non-constant along $S_{a,b}$ produce width at least $b-a$, so the only possibility for width one would be $a=0$ and $b=1$. It is clear that $\widetilde F_2(0,1)$ and $\widetilde F_3(0,1,1)$ have width one with respect to the functional $z$ but it is easy to check (and left to the reader) that the rest have width at least two even when $(a,b)=(0,1)$. 
\end{proof}

The list in Lemma~\ref{lemma:inf_families} contains some redundancy, but not much.

\begin{proposition}
\label{prop:redundancy}
The only isomorphic polytopes among the list of Lemma~\ref{lemma:inf_families} are:
\begin{enumerate}

\item For $i\in\{1,2,4\}$ we have that $\widetilde F_i(a,b)$ is isomorphic to $\widetilde F_i(-b,-a)$.
\item $\widetilde F_3(a,b,k)$ is isomorphic to $\widetilde F_3(k-b,k-a,k)$. 
\item In size six:
\[
\widetilde F_2(0,2)\cong\widetilde F_4(0,1) \quad\text{ and }\quad
\widetilde F_2(-1,1)\cong\widetilde F_3(0,1,0).
\] 
\item In size seven:
\[
\widetilde F_3(0,2,1)\cong\widetilde F_4(-1,1).
\]
\end{enumerate}
\end{proposition}

\begin{proof}
The following maps show the isomorphisms in parts (1) and (2) among polytopes within each family:
\begin{align*}
(x,y,z) \mapsto (y,x,-z) \quad& \Rightarrow \quad \widetilde F_1(a,b)\cong \widetilde F_1 (-b,-a).\\
(x,y,z) \mapsto (x,y,-x-y-z) \quad& \Rightarrow \quad \widetilde F_2(a,b)\cong \widetilde F_2 (-b,-a).\\
(x,y,z) \mapsto (y,-x, k+y(k-1)-z ) \quad& \Rightarrow \quad \widetilde F_3(a,b,k)\cong \widetilde F_3 (k-b,k-a,k).\\
(x,y,z) \mapsto (x,y,-x-y-z) \quad& \Rightarrow \quad \widetilde F_4(a,b)\cong \widetilde F_4 (-b,-a).
\end{align*}
Whenever $b-a\ge 3$ these are the only isomorphisms, by the following arguments: 
\begin{itemize}
\item The spike contains $b-a+1\ge 4$ collinear lattice points and is the only collinearity of at least four lattice points. Thus, every isomorphism sends the spike to the spike and, in particular, $b-a$ is an invariant (modulo unimodular equivalence).
\item Polytopes in different families $\widetilde F_i$ cannot be isomorphic: the polytopes in the family $\widetilde F_1$ are the only ones of index three; those in $\widetilde F_2$ are the only ones of index two with only three points outside the spike; and those in $\widetilde F_4$ are the only ones with three collinear lattice points outside the spike.
\item For $i=1,2,4$ the pair $\{-a,b\}$ is an invariant since the plane spanned by lattice points not in the spike intersects the spike at the point $(0,0,0)$, which is at distances $-a$ and $b$ of the end-points of the spike.
\item For $i=3$ the pair $\{-a,b-k\}$ is an invariant since the midpoints of the pairs of opposite lattice points around the spike have as midpoints $(0,0,0)$ and $(0,0,k)$, which are on the spike and at distances $-a$ and $b-k$ from the end-points.
\end{itemize}

It only remains to see that, when $b-a \le 2$, the only isomorphisms that appear are those of parts (3) and (4) of the statement. 
By parts (1) and (2) we can assume that $-a \le b$ in all cases and that $-a \le b-k$ in the family $\widetilde F_3$. That is, $(a,b)\in \{(0,1),(0,2),(-1,1)\}$ in $\widetilde F_1$, $\widetilde F_2$ and $\widetilde F_4$, and $(a,b,k)\in\{(0,1,0),(0,2,0),(0,2,1),(0,2,2),(-1,1,0)\}$ for $\widetilde F_3$ (remember that $\widetilde F_2(0,1)$ and $\widetilde F_3(0,1,1)$ are excluded).
Let us separate the cases by index and size:

\begin{itemize}

\item In the family $\widetilde F_1$, the only one of index $3$, the three possibilities are distinguished by the fact that $\widetilde F_1(0,1)$ is has size five, $\widetilde F_1(0,2)$ is a tetrahedron of size six, and $\widetilde F_1(-1,1)$ is a triangular bipyramid of size six.

\item For index $2$ and size six there are only $\widetilde F_2(-1,1)$, $\widetilde F_3(0,1,0)$, $\widetilde F_2(0,2)$ and $\widetilde F_4(0,1)$. 
The first two are square pyramids, and isomorphic via $(x,y,z) \mapsto (x+z,y+z,-x-y-z)$. The last two are tetrahedra, isomorphic via $(x,y,z) \mapsto (-x+y-z+1,z-1,-y)$. 
 
\item For index $2$ and size seven there are four possibilities in the family $\widetilde F_3$ and two in the family $\widetilde F_4$. All of them happen to have three collinear triples of lattice points.
Looking at the distribution of the seven lattice points in collinear triples is enough to distinguish among five of the six possibilities, as the following diagram shows. (Each diagram shows what lattice points form collinear triples but also the relative order of points along each triple):
\smallskip

\centerline{\includegraphics[scale=.54]{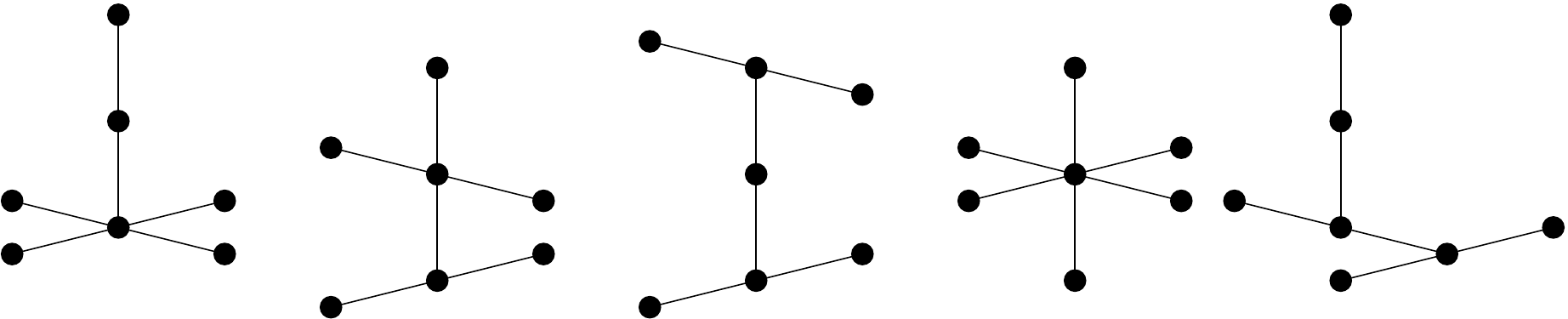}}
\centerline{
$\widetilde F_3(0,2,0)$, \quad
$\widetilde F_3(0,2,1)$, \quad
$\widetilde F_3(0,2,2)$, \quad
$\widetilde F_3(-1,1,0)$, \quad
$\widetilde F_4(0,2)$. \quad
}
\medskip
The unimodular transformation $(x,y,z) \mapsto (-x+y-z+1,z-1,x)$ maps $\widetilde F_3(0,2,1)$ to the remaining polytope of size seven, $\widetilde F_4(-1,1)$.

\end{itemize}
\end{proof}

\begin{corollary}
\label{coro:numbers}
The number of isomorphism classes of polytopes in the families $\widetilde F_i$ are as given in Table~\ref{table:families}.

\begin{table}[htb]\rm
\begin{center}
\begin{tabular}{r|rrrrrrr|c}
\multicolumn{1}{r|}{\bf size}      &  {\bf 5}   &  {\bf  6}  &    {\bf 7}  &   {\bf 8}  &   {\bf 9}  &{\bf 10}  &{\bf 11}   &   \textrm{\bf general $n$}\\
\hline
family $\widetilde F_1$ & 1   &  2   &  2  & 3  &   3  &   4   &   4    & $\lceil \frac{n-3}{2}\rceil$  \rule[-1.5ex]{0pt}{4.5ex} \\
\hline
{\bf total index 3}  &\bf1 & \bf2 &\bf2 &\bf3 &\bf3  &\bf4& \bf4  & $\lceil \frac{n-3}{2}\rceil$ \rule[-1.5ex]{0pt}{4ex} \\
\hline
\hline
family $\widetilde F_2$ & $0^*$   &  2   &  2  & 3  &   3  &   4   &   4   &   $\lceil \frac{n-3}{2}\rceil$  \rule{0pt}{3ex}\\
family $\widetilde F_3$ &    &   $1^*$ &    4  &    6   &   9   &  12   &   16     &  $\left\lfloor \left(\frac{n-3}{2}\right)^2\right\rfloor$\rule{0pt}{3ex}\\
family $\widetilde F_4$ &   &  1   &  2   &  2  & 3  &  3   &   4   & $\left\lfloor \frac{n-3}{2}\right\rfloor$  \rule[-2ex]{0pt}{5ex}\\
\hline
{\bf total index 2}  & \bf0   &   \bf2$^{**}$  &  \bf7$^{**}$  &  \bf11   & \bf15   &\bf19   & \bf24      &  $\left\lfloor \frac{(n-3)(n+1)}{4}\right\rfloor$\rule[-1ex]{0pt}{4.5ex}\\
\end{tabular}
\end{center}
\medskip
\caption{The number of non-spanning $3$-polytopes in the infinite families of Lemma~\ref{lemma:inf_families}.
The two entries marked with $^{*}$ do not coincide with the formula for general $n$ because in each of them one of the configurations counted by the formula has width one. The entries marked $^{**}$ are less than the sum of the three above them (and, in particular, do not coincide with the general formula) because of the isomorphisms in parts (3) and (4) of Proposition~\ref{prop:redundancy}.
}
\label{table:families}
\end{table}
\end{corollary}

\begin{proof}
For size up to seven, the counting is implicit in the proof of parts (3) and (4) of Proposition~\ref{prop:redundancy}. Thus, in the rest of the proof we assume size $n\ge 8$ and the only isomorphisms we need to take into account are those in parts (1) and (2) of Proposition~\ref{prop:redundancy}.

In the families $\widetilde F_1$ and $\widetilde F_2$ we have $b-a+1 = n-3$ lattice points along the spike and apart of size the only invariant is on which of them does the plane of the other lattice points intersect. This gives $n-3$ possibilities, but opposite ones are isomorphic so the count is $\lceil (n-3)/2\rceil$. For $\widetilde F_4$ we have the same count except that now $b-a+1=n-4$, so we get $\lceil (n-4)/2\rceil= \lfloor (n-3)/2\rfloor$.

For $\widetilde F_3$ we have $b-a+1=n-4$ lattice points along the spike, and we have to choose (perhaps with repetition) two of them to be mid-points of non-spike pairs of lattice points. This gives $\binom{n-3}{2}$ possibilities.
To mod out symmetric choices we divide that number by two but then have to add one half of the self-symmetric choices, of which there are $\lceil (n-4)/2\rceil= \lfloor (n-3)/2\rfloor$. The count, thus, is
\[
\frac12 \left(\binom{n-3}{2} + \left\lfloor \frac{n-3}2\right\rfloor\right) = 
\frac12 \left\lfloor \frac{n^2 - 6n-9}2\right\rfloor =
 \left\lfloor \frac{(n - 3)^2}4\right\rfloor.
\]
\end{proof}


\section{Proof of Theorem~\protect{\ref{thm:main_exceptions}}}
\label{sec:main}

\subsection{Merged and non-merged $3$-polytopes}
\label{sec:merging}

For a lattice $d$-polytope $P\subset \R^d$ of size $n$ and a vertex $v$ of $P$ we denote
\[
P^v := \conv (P\cap \Z^d \setminus\{v\}),
\]
which has size $n-1$ and dimension $d$ or $d-1$. We abbreviate $(P^u)^v$ as $P^{uv}$.
The following definition is taken from~\cite{quasiminimals}.

\begin{definition}
Let $P$ be a lattice $3$-polytope of width $>1$ and size $n$. We say that $P$ is \emph{merged} if there exist at least two vertices $u,v\in\wert(P)$ such that $P^u$ and $P^v$ have width larger than one and such that $P^{uv}$ is $3$-dimensional.
\end{definition}

Loosely speaking, we call a polytope of size $n$ merged if it can be obtained \emph{merging} two subpolytopes $Q_1\cong P^u$ and $Q_2\cong P^v$ of size $n-1$ and width $>1$ along their common (full-dimensional) intersection $P^{uv}$. This merging operation is the basis of the enumeration algorithm in~\cite{quasiminimals} and to make it work, a complete characterization of the polytopes that are not merged was undertaken. 
Combining several results from~\cite{quasiminimals} we can prove that:

\begin{theorem}
\label{thm:spiked-index}
Let $P$ be a lattice $3$-polytope of size $n\ge 8$ and suppose that $P$ is not merged.
Then, one of the following happens:

\begin{enumerate}
\item $P$ contains a unimodular tetrahedron, and in particular it is spanning.

\item $P$ has sublattice index $2$ and is isomorphic to
\[
\widetilde F_2(0,n-4) = \conv\{(1, -1, 0), (-1, 1, 0), (-1, -1, 1), (0, 0, n-4)\}.
\]

\item $P$ has sublattice index $3$ and is isomorphic to
\[
\widetilde F_1(0,n-4)=\conv\{(2, -1, 1), (-1, 2, -1), (-1, -1, 0), (0, 0, n-4)\}.
\]
\end{enumerate}
\end{theorem}

\begin{proof}
This follows from Theorems~2.9, 2.12, 3.4 and 3.5 in~\cite{quasiminimals}, but let us briefly explain how. 
\begin{itemize}
\item \cite[Theorem~2.12]{quasiminimals} shows that every non-merged $3$-polytope of size $\ge 7$ is quasi-minimal, according to the following definition: it has at most one vertex $v$ such that $P^v$ has width larger than one. 

\item \cite[Theorem 2.9]{quasiminimals} says that every quasi-minimal $3$-polytope is either \emph{spiked} or \emph{boxed}. 

\item By~\cite[Proposition 4.6]{quasiminimals}, boxed polytopes of size $\ge 8$ contain $n-3\ge5$ vertices of the unit cube, so in particular they contain a unimodular tetrahedron.
\end{itemize}

Thus, for the rest of the proof we can assume that $P$ is spiked. Theorems 3.4 and 3.5 in~\cite{quasiminimals} contain a very explicit description of all spiked $3$-polytopes, which in particular implies that every spiked $3$-polytope $P$ of size $n$ projects to one of the ten $2$-dimensional configurations $A'_1,\dots,A'_{10}$ of Figure~\ref{fig:spiked_projections} (notation taken from~\cite{quasiminimals}), where the number next to each point in $A'_i$ indicates the number of lattice points in $P$ projecting to it. 
\begin{figure}[htb]
\centerline{\includegraphics[scale=.7]{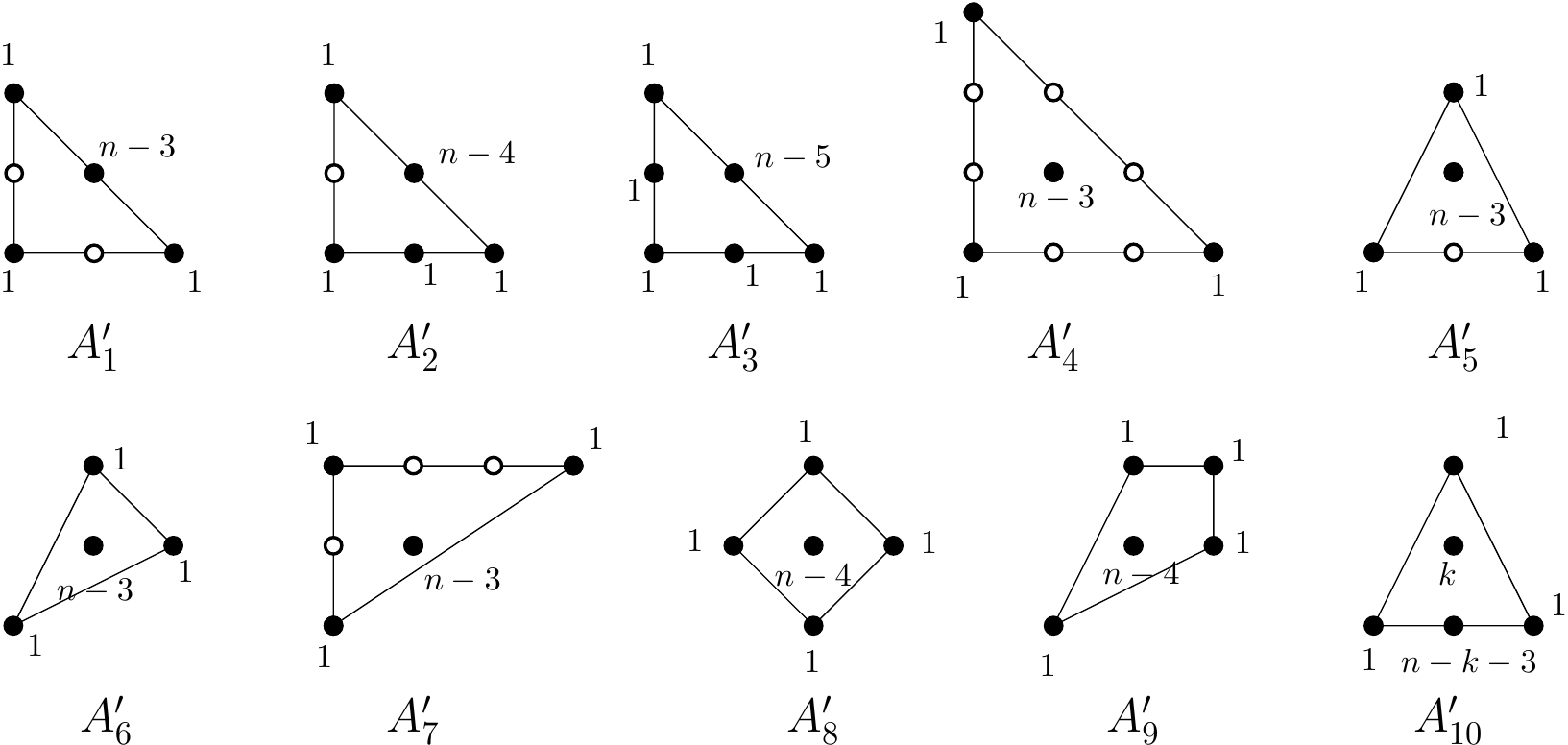}}
\caption{Possible projections of a non-merged $3$-polytope $P$ of size $n\ge 8$ along its spike. A black dot with a label $j$ is the projection of exactly $j$ lattice points in $P$.
White dots are lattice points in the projection of $P$ which are not the projection of any lattice point in $P$.}
\label{fig:spiked_projections}
\end{figure}

An easy inspection shows that all $A'_i$ except $A'_1$ and $A'_4$ have a unimodular triangle $T$ with the property that (at least) one vertex of $T$ has at least two lattice points of $P$ in its fiber. Then $T$ is the projection of a unimodular tetrahedron in $P$. 
It thus remains only to show that projections to $A'_1$ and $A'_4$ correspond exactly to cases (2) and (3) in the statement. For this we use the coordinates shown in Figure~\ref{fig:spiked_nonspanning}.
Observe that $A'_1$ and $A'_4$ are exactly the $F_2$ and $F_1$ of Lemma~\ref{lemma:inf_families}, respectively.

\begin{figure}[htb]
\centerline{\includegraphics[scale=.7]{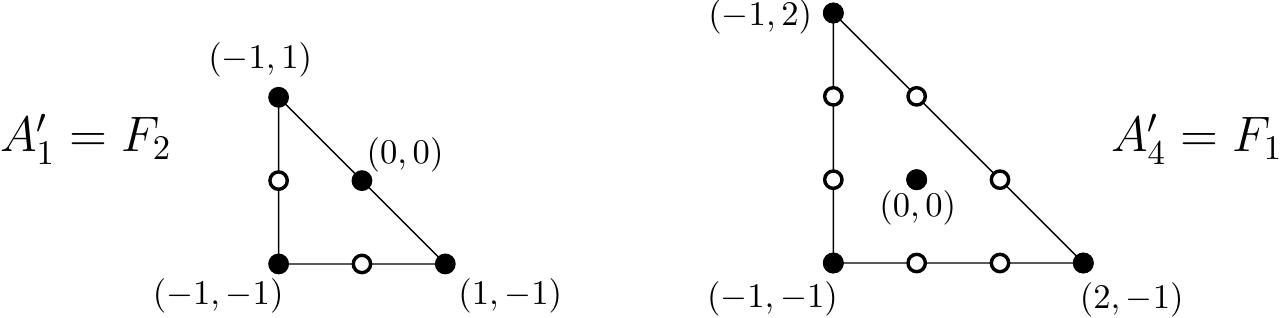}}
\caption{Possible projections of non-merged non-spanning $3$-polytopes of size $\ge 8$.}
\label{fig:spiked_nonspanning}
\end{figure}

\begin{itemize}
\item Suppose that $P$ projects to $F_2$.
By the assumption on the number of lattice points of $P$ projecting to each point in $F_2$, $P$ has exactly three lattice points outside the spike $\{(0,0,t): t \in \R\}$, and they have coordinates $(1,-1, h_1)$, $(-1,1, h_2)$ and $(-1,-1,h_3)$. In order not to have lattice points in $P$ projecting to the white dots in the figure for $F_2$, $h_1$ and $h_2$ must be of the same parity and $h_3$ of the opposite parity. There is no loss of generality in assuming $h_1$ and $h_2$ even, in which case the unimodular transformation 
\[
(x,y,z) \to \left(x,y, \frac{h_3-h_1-1}2 x + \frac{h_3-h_2-1}2 y + z - \frac{h_1+h_2}2\right)
\]
sends $P$ to $\widetilde F_2(a,b)$ for some $a\le 0 <b$ and $b-a=n-4$. 

If both $a$ and $b$ are non-zero, then $(0,0,a)$ and $(0,0,b)$ are vertices of $P$ and $P$ is merged from $P^{(0,0,a)}$ and $P^{(0,0,b)}$ (of width $>1$ since $n\ge 8$), which is a contradiction. Thus, one of $a$ or $b$ is zero, and by the isomorphism in part (1) in Proposition~\ref{prop:redundancy} there is no loss of generality in assuming it to be $a$. 

\item For $F_1$ the arguments are essentially the same, and left to the reader.
\end{itemize}
\end{proof}

\begin{corollary}
\label{coro:width>1-index}
With the only exception of the tetrahedron $E_{(5,5)}$ of Theorem~\ref{thm:main_exceptions} (of size $5$ and index $5$) every lattice $3$-polytope of width $>1$ has index at most $3$.
\end{corollary}

\begin{proof}
Let $P$ be a lattice $3$-polytope of width $>1$. If $P$ has size at most $7$ (or at most $11$, for that matter), the statement follows from the enumerations in~\cite{quasiminimals}, as seen in Table~\ref{table:numbers}. For size $n\ge 8$ we use induction, taking $n=7$ as the base case. 
Either $P$ is non-merged, in which case Theorem~\ref{thm:spiked-index} gives the statement, or $P$ is merged, in which case it has a vertex $u$ such that $P^u$ still has width $>1$. By inductive hypothesis $P^u$ has index at most three, and by Lemma~\ref{lemma:index}(\ref{item:subpolytope}), the index of $P$ divides that of $P^u$.
\end{proof}


\subsection{Lattice $3$-polytopes of index $3$}
\label{sec:index3}

\begin{theorem}
\label{thm:index3}
Let $P$ be a lattice $3$-polytope of width $>1$ and of index three. Then $P$ is equivalent to 
either the tetrahedron $E_{(6,3)}$ of Theorem~\ref{thm:main_exceptions} (of size six) or to a polytope in the family $\widetilde F_1$ of Lemma~\ref{lemma:inf_families}.
\end{theorem}

\begin{proof}
Let $n$ be the size of $P$.
The statement is true for $n\le 11$ by the enumerations in~\cite{quasiminimals}, as seen comparing Tables~\ref{table:numbers} and~\ref{table:families}. For size $n > 11$ we use induction, taking $n=11$ as the base case. 

Let $n>11$. If $P$ is not merged, then the result holds by Theorem~\ref{thm:spiked-index}. So for the rest of the proof we assume that $P$ is merged. Then there exist vertices $u,v\in\wert(P)$ such that $P^u$ and $P^v$ (of size $n-1$) have width $>1$, and such that $P^{uv}$ (of size $n-2$) is full-dimensional. 
By Lemma~\ref{lemma:index}(\ref{item:subpolytope}) the sublattice indices of $P^u$ and $P^v$ are multiples of $3$ and by Corollary~\ref{coro:width>1-index} they equal $3$. Thus, by induction hypothesis, both $P^u$ and $P^v$ are in the family $\widetilde F_1$.

Since both $P^u$ and $P^v$ have a spike with $n-4>7$ lattice points and the spike is the only collinearity of more than three lattice points, $P^u$ and $P^v$ have their spikes along the same line. 
If $u\in P^v$ lies along the spike of $P^v$ (resp. if $v$ lies along the spike of $P^u$), then $P$ is obtained from $P^u$ (resp. $P^v$) by extending its spike by one point, which implies $P$ is also in the family $\widetilde F_1$.

Hence, we only need to study the case where both $u$ and $v$ are outside the spike of $P^v$ and $P^u$, respectively.

In this case, $P^{uv}$ consists of a spike of length $n-5$ plus two lattice points outside of it. Let $w_1$ and $w_2$ be these two lattice points.
We use the following observation about the polytopes in the family $\widetilde F_1$: \emph{the barycenter of the three lattice points outside the spike is a lattice point in the spike}. That is, we have that both $\frac13(u+w_1+w_2)$ and $\frac13(v+w_1+w_2)$ are lattice points in the spike. In particular, $u-v$ is parallel to the spike and of length a multiple of three, which is a contradiction to the fact that $u$ and $v$ are the only lattice points of $P$ not in $P^{uv}$. See Figure~\ref{figure:mergingsF_1}.
\begin{figure}[htb]
\centerline{\includegraphics[scale=.7]{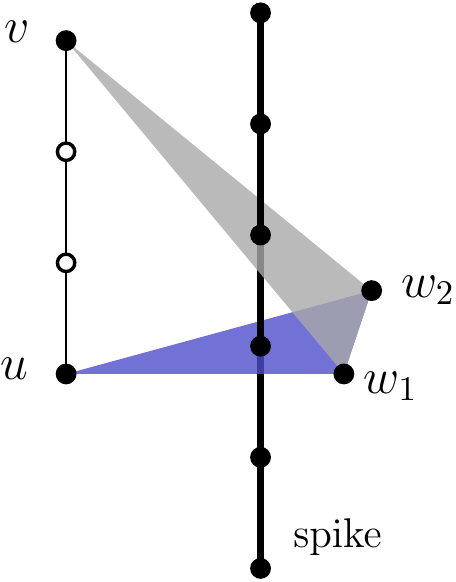}}
\caption{The setting at the end of the proof of Theorem~\ref{thm:index3}}
\label{figure:mergingsF_1}
\end{figure}
\end{proof}


\subsection{Lattice $3$-polytopes of index $2$}
\label{sec:index2}

The case of index $2$ uses the same ideas as in index $3$ but is a bit more complicated, since there are more cases to study. 
As a preparation for the proof, the following two lemmas collect properties of the polytopes $\widetilde F_i(a,b,\_)$ of index two.

\begin{lemma}
\label{lemma:auto_families}
\begin{enumerate}
\item Let $Q=\widetilde F_3(a,b,k)$ and $w$ be any vertex of $Q$ not in the spike. Then $Q^w$ is isomorphic to either $\widetilde F_2(a,b)$ or $\widetilde F_2(k-b,k-a)$.
\label{item:auto3}

\item Let $Q=\widetilde F_4(a,b)$ and $w\in\{(-1,-1,1),(3,-1,-1)\}$. Then $Q^w\cong\widetilde F_2(a,b)$.
\label{item:auto4}
\end{enumerate}

\end{lemma}
\begin{proof}
For part (2) simply observe that $(x,y,z) \mapsto (-x-2y,y,x+y+z)$ is an automorphism of $F_4(a,b)$ that swaps $(-1,-1,1)$ and $(3,-1,-1)$. Thus $Q^{(-1,-1,1)}\cong Q^{(3,-1,-1)} =\widetilde F_2(a,b)$.

In part (1), the automorphism of $\widetilde F_3(a,b,k)$
\[
(x,y,z) \mapsto (-y,-x,-(k-1)x-(k-1)y+z)
\]
swaps $(-1,-1,1)$ and $(1,1,2k-1)$, and gives 
\[
Q^{(-1,-1,1)}\cong Q^{(1,1,2k-1)} = \widetilde F_2(a,b).
\]
But we have also the isomorphism in part (2) of Proposition~\ref{prop:redundancy}, which sends $Q$ to $\widetilde F_3(k-b,k-a,k)$ and maps $\{(-1,-1,1), (1,1,2k-1)\}$ to $\{(1,-1,0), (-1,1,0)\}$. Thus
\[
Q^{(1,-1,0)}\cong Q^{(-1,1,0)}\cong \widetilde F_2(k-b,k-a).
\]

\end{proof}

\begin{lemma}
\label{lemma:spike4}
Let $Q$ be of size $\ge 8$ and isomorphic to a polytope in one of the families $\widetilde F_3$ or $\widetilde F_4$ of Lemma~\ref{lemma:inf_families}. Let $w$ be a vertex of $Q$ such that $Q^w=\widetilde F_2(a,b)$. 
Then, $w$ is one of $(3,-1,-1)$, $(-1,3,-1)$, or $(1,1,2j-1)$ for some $j\in \{a,\dots,b\}$.
\end{lemma}

\begin{figure}[htb]
\centerline{\includegraphics[scale=.8]{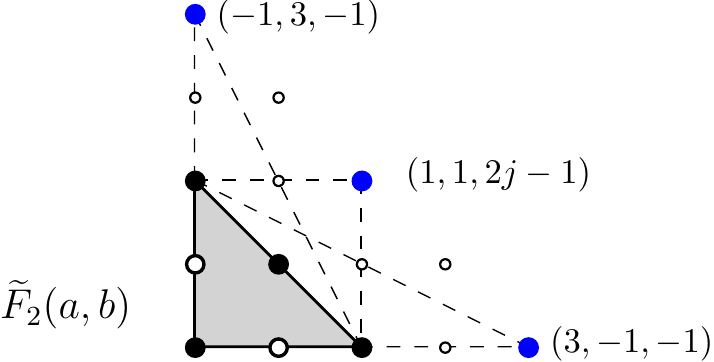}}
\caption{Illustration of Lemma~\ref{lemma:spike4}.}
\label{fig:3vertices}
\end{figure}

\begin{proof}
Since $Q$ has four lattice points outside the spike and $Q^w$ has three, their spikes have the same size. Also, since the size of $Q$ is $\ge 8$, the spike has at least four lattice points. In particular $Q$ and $Q^w$ have the same spike.
The four points of $Q$ outside the spike are $(1,-1,0)$, $(-1,1,0)$, $(-1,-1,1)$ and $w$. The fact that $(1,-1,0)$, $(-1,1,0)$ have their midpoint on the spike implies that:
\begin{itemize}
\item 
 If $Q$ is in the family $\widetilde F_3$ then the midpoint of $w$ and $(-1,-1,1)$ must also be a lattice point in the spike. That is, $w=(1,1,2j-1)$ with $j\in \{a,\dots, b\}$. 
\item 
If $Q$ is in the family $\widetilde F_4$ then $w$ and $(-1,-1,1)$ have as midpoint one of $(1,-1,0)$ and $(-1,1,0)$, so that $w\in \{(3,-1,-1),(-1,3,-1)\}$.
\end{itemize}
\end{proof}

\begin{theorem}
\label{thm:index2}
Let $P$ be a lattice $3$-polytope of width $>1$ and of index two. Then $P$ is equivalent to either one of the four exceptions of index two in Theorem~\ref{thm:main_exceptions} or to one of the polytopes in the families $\widetilde F_2$, $\widetilde F_3$ or $\widetilde F_4$ of Lemma~\ref{lemma:inf_families}. 
\end{theorem}

\begin{proof}
With the same arguments as in the proof of Theorem~\ref{thm:index3} we can assume that $P$ has size $> 11$, is merged from $P^u$ and $P^v$, that $P^u$ and $P^v$ are in the families $\widetilde F_2$, $\widetilde F_3$ or $\widetilde F_4$ and have the same spike, and that $u$ and $v$ lie outside the spike.
In particular, $P^u$ and $P^v$ have the same number of lattice points outside the spike. This number can be three or four, and we look at the two possibilities separately:

\begin{itemize}
\item If $P^u$ and $P^v$ have three lattice points not on the spike then they are both in the family $\widetilde F_2$. The lattice points in $P^{uv}$ are those in the spike (which is common to $P^u$ and $P^v$) together with two extra lattice points $w_1$ and $w_2$. 

Let $\phi_u: \widetilde F_2(a,b) \to P^u$ and $\phi_v: \widetilde F_2(a',b')\to P^v$ be isomorphisms. 
Then $\phi_u(-1,-1,1) \ne v$ and $\phi_v(-1,-1,1) \ne u$, since this would imply $P^{uv}$ to be $2$-dimensional. Thus, $\phi_u(-1,-1,1), \phi_v(-1,-1,1) \in \{ w_1,w_2\}$. This implies
\begin{align*}
\{\phi_u(-1,1,0), \phi_u(1,-1,0)\} &= \{v,w_i\}, \text{ and }\\
\{\phi_v(-1,1,0), \phi_v(1,-1,0)\} &= \{u,w_j\} 
\end{align*}
for some $i,j\in \{1,2\}$.
In particular the segments $vw_i$ and $uw_j$ have as mid-points the points $\phi_u(0,0,0)$ and $\phi_v(0,0,0)$, respectively, which are lattice points along the spike.
This gives two possibilities for the merging, illustrated in Figure~\ref{figure:mergingsF_2}: if $i=j$ (top part of the figure) then $uv$ is parallel to the spike and of even length, which gives a contradiction: the midpoint of $u$ and $v$ would be a third lattice point of $P$ outside $P^{uv}$.
If $i\ne j$ then $P^u \cup P^v$ are the lattice points in a polytope isomorphic to $\widetilde F_3(a,b,k)$ (bottom row in the figure), so the statement holds.

\begin{figure}[htb]
\centerline{
\begin{tabular}{|cc|c|c|}
\hline
$P^{uv}\subset P^u$ & $P^{uv}\subset P^v$ & $P^u \cup P^v$ & OUTCOME\\
\hline
	\includegraphics[scale=.8]{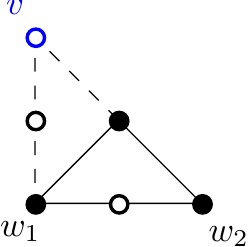} &
	\includegraphics[scale=.8]{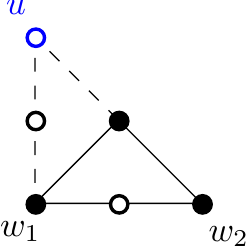} &
	\includegraphics[scale=.8]{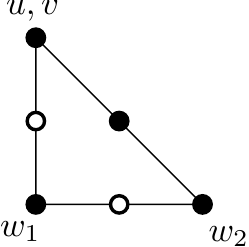}\rule{0pt}{14.5ex} & \raisebox{6ex}{contradiction}\\
\hline
	\includegraphics[scale=.8]{ind2_lv.pdf} & 
	\includegraphics[scale=.8]{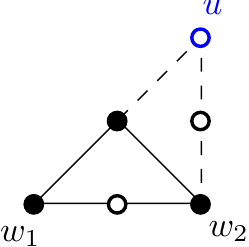} &
	\includegraphics[scale=.8]{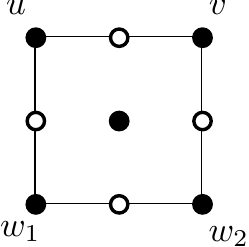}\rule{0pt}{14ex}& \raisebox{6ex}{$\widetilde F_3(a,b,k)$}\\
\hline
\end{tabular}
}
\caption{The two cases in the first part of the proof of Theorem~\ref{thm:index2}.}
\label{figure:mergingsF_2}
\end{figure}

\item If $P^u$ and $P^v$ have four lattice points not on the spike then they belong to the families $\widetilde F_3$ or $\widetilde F_4$.

Our first claim is that there is no loss of generality in assuming that $P^{uv}$ is in the family $\widetilde F_2$. Indeed, if $P^u$ is in the family $\widetilde F_3$ then this is automatically true by Lemma~\ref{lemma:auto_families}\eqref{item:auto3}. If $P^u= \widetilde F_4(a,b)$ then at least one of the vertices $(-1,-1,1)$ and $(3,-1,-1)$ of $P^u$ is also a vertex of $P$ (e.g., because these two vertices are the unique minimum and maximum on $P^u$ of the linear functional $f(x,y,z)=x+y$).
Call that vertex $v'$, and consider $P$ as merged from $P^u$ and $P^{v'}$ instead of the original $P^u$ and $P^v$. Then $P^{uv'}$ is in the family $\widetilde F_2$ by Lemma~\ref{lemma:auto_families}\eqref{item:auto4}.

Thus, we can assume for the rest of the proof that $P^{uv}= \widetilde F_2(a,b)$. Lemma~\ref{lemma:spike4} (applied first with $Q=P^u$ and $w=v$ and then with $Q=P^v$ and $w=u$) implies that both $u$ and $v$ belong to 
\[
\{(3,-1,-1), (-1,3,-1)\}\cup\{ (1,1,2j-1) : j=a,\dots, b\}
\]
which is contained in $\{(x,y,z) :\ x+y=2, \ x\equiv y\equiv z\equiv 1 \pmod 2\}$.
But then the segment $uv$, lying in the hyperplane $\{x+y=2\}$, contains at least an extra lattice point, namely $\frac12(u+v)$. See Figure~\ref{fig:3vertices}.
This implies $P$ has at least three more lattice points than $P^{uv}$, a contradiction.
\end{itemize}
\end{proof}


\section{(Almost all) spanning $3$-polytopes have a unimodular tetrahedron}
\label{subsec:primitiveIFF}
\label{sec:spanning}

The following statement follows from Corollary~\ref{coro:volume_i}.

\begin{lemma}
\label{lemma:subpolytopeofindex>1}
Let $P$ be a non-spanning lattice $3$-polytope and let $v$ be a vertex of it such that $P^v$ is still $3$-dimensional.
Then $P$ and $P^v$ have the same index.
\qed
\end{lemma}

With this we can now prove Theorem~\ref{thm:no_uni_tetra}:
The only spanning $3$-polytopes that do not contain a unimodular tetrahedron are the $E_{(5,1)}^1$ and $E_{(5,1)}^2$ from the statement of Theorem~\ref{thm:no_uni_tetra}.

\begin{proof}[Proof of Theorem~\ref{thm:no_uni_tetra}]
Let $P$ be a spanning $3$-polytope.
If $P$ has width one, the statement is true by Corollary~\ref{cor:width1-iff}. So assume $P$ to be of width $>1$ and let $n$ be its size.
If $n\le 7$ then the statement is true by the enumerations in~\cite{quasiminimals}, and if $P$ is not merged and of size $n\ge 8$ by Theorem~\ref{thm:spiked-index}. So, we suppose that $n\ge8$ and $P$ is merged. That is, there exist $u,v\in\wert(P)$ such that $P^u$ and $P^v$ have width $>1$, and such that $P^{uv}$ is $3$-dimensional. 

If $P^u$ or $P^v$ are spanning, then by inductive hypothesis they contain a unimodular tetrahedron, and so does $P$. 
To finish the proof we show that it is impossible for $P^u$ and $P^v$ to both have index greater than one. If this happened, then Lemma~\ref{lemma:subpolytopeofindex>1} tells us that both $P^u$ and $P^v$ have the same index as $P^{uv}$. That is, $u$ and $v$ lie in the affine lattice spanned by $P^{uv}\cap \Z^3$, which implies the index of $P$ being the same (and bigger than one), a contradiction.
\end{proof}


\section{The $h^*$-vectors of non-spanning $3$-polytopes}
\label{sec:hstar}

Part of our motivation for studying non-spanning $3$-polytopes comes from results on $h^*$-vectors of \emph{spanning} lattice polytopes recently obtained by Hofscheier et al.~in~\cite{HKN-a,HKN-b}. 
Recall that the \emph{$h^*$-polynomial} of a lattice $d$-polytope $P$, first introduced by Stanley~\cite{Stanley}, is the numerator of the generating function of the sequence $(\operatorname{size}(tP))_{t\in \N}$. That is, it is the polynomial
\[
{h^*_0+ h^*_1 z + \dots + h^*_s z^s}:= (1-z)^{d+1} \sum_{t=0}^\infty \operatorname{size}(tP)\, z^t.
\]
Its coefficient vector $h^*(P):=(h^*_0, \dots, h^*_s)$ is the \emph{$h^*$-vector} of $P$. All entries of $h^*(P)$ are known to be nonnegative integers, and the degree $s\in \{0,\dots,d\}$ of the $h^*$-polynomial is called the \emph{degree} of $P$. 
See~\cite{BeckRobins} for more details.

For a lattice $d$-polytope with $n$ lattice points in total, $n_0$ of them in the interior, and with normalized volume $V$, one has $h^*_0=1$, $h^*_1= n-d-1$, $h^*_{d}=n_0$ and $\sum_i h^*_i = V$. In particular, in dimension three the $h^*$-vector can be fully recovered from the three parameters $(n,n_0,V)$ as follows:
\begin{equation}
\label{eq:h}
h^*_0=1,\qquad
h^*_1= n-4,\qquad
h^*_2=V +3 -n_0 -n,\qquad
h^*_3= n_0.
\end{equation}

From this we can easily compute the $h^*$-vectors of all non-spanning $3$-polytopes. They are given in Table~\ref{table:hstar}. For their computation, in the infinite families $\widetilde F_i$ we use that:
\begin{itemize}
\item The total number of lattice points equals $b-a+4$ in $\widetilde F_1(a,b)$ and $\widetilde F_2(a,b)$ and it equals $b-a+5$ in $\widetilde F_3(a,b,k)$ and $\widetilde F_4(a,b)$.

\item The number of interior lattice points is zero in $\widetilde F_2(a,b)$ and $b-a-1$ in the other three.

\item The volume always equals the volume of the projection $F_i$ (see Figure~\ref{fig:projs_families}) times the length $b-a$ of the spike. This follows from Corollary~\ref{coro:volume_i}.
\end{itemize}

\begin{table}[htb]\rm
\[
\begin{array}{l|rrrr}
\quad\ \ P   & (h^*_0, & h^*_1, & h^*_2, & h^*_3) \rule[-1.5ex]{0pt}{3ex}\\
\hline
T_{p,q}(a,b)  & (1,&a+b-2,&abq-a-b+1, &0)\rule{0pt}{3ex}\\
\widetilde F_1(a,b)  & (1,& n-4, &7n-28,& n-5) \rule{0pt}{3ex}\\
 \widetilde F_2(a,b)  & (1,&  n-4 , &3n-13,&  0) \rule{0pt}{3ex}\\ 
 \widetilde F_3(a,b,k)   & (1,& n-4, &6n-31,& n-6)\rule{0pt}{3ex} \\
 \widetilde F_4(a,b)   & (1,& n-4, &6n-31,& n-6) \rule{0pt}{3ex}\\
 \\
\end{array}
\]
\[
\begin{array}{c|cccc}
P    & (h^*_0, & h^*_1, & h^*_2, & h^*_3) \rule[-1.5ex]{0pt}{3ex}\\
\hline
E_{(5,5)}     & (  1 ,&1  ,& 17 ,& 1 ) \rule{0pt}{2.5ex}\\
E_{(6,3)}     & (  1 ,&2  ,& 19 ,& 2 )\rule{0pt}{2.5ex} \\
E_{(7,2)}     & (  1 ,&3  ,& 10 ,& 0 )\rule{0pt}{2.5ex} \\
E_{(8,2)}^1 & (  1 ,&4  ,& 20 ,& 3 )\rule{0pt}{2.5ex} \\
E_{(8,2)}^2 & (  1 ,&4  ,& 20 ,& 3 )\rule{0pt}{2.5ex} \\
E_{(8,2)}^3 & (  1 ,&4  ,& 21  ,&4  )\rule{0pt}{2.5ex} \\
\end{array}
\]
\medskip
\caption{The $h^*$-vectors of non-spanning $3$-polytopes. In the infinite families (top table) $n$ is the size of the polytope.
Remember that: in $T_{p,q}(a,b)$ we have $a,b\ge 1$, $n=a+b+2$ and $q>1$;
in $\widetilde F_1(a,b)$ we have $n\ge 5$ and in the rest $n\ge 6$.}
\label{table:hstar}
\end{table}

The main results in~\cite{HKN-a,HKN-b} are:

\begin{theorem}[\protect{\cite[Theorem 1.3]{HKN-a}}, \protect{\cite[Theorem 1.2]{HKN-b}}]
\label{thm:HKN}
Let $h^*=(h^*_0,\dots,h^*_s)$ be the $h^*$-vector of a spanning lattice $d$-polytope of degree $s$. Then:
\begin{enumerate}
\item $h^*$ has no gaps, that is, $h^*_i > 0$ for all $i\in \{0,\dots,s\}$.
\item for every $i,j\ge 1$ with $i+j<s$ one has
\[
h^*_1 + \dots + h^*_i \le h^*_{j+1} + \dots + h^*_{j+i}.
\]
\end{enumerate}
\end{theorem}

In dimension three the only nontrivial inequality in part (2) is $h^*_1\le h^*_2$ for lattice polytopes of degree $s=3$ (that is, for lattice polytopes with interior lattice points), which 
is true by Hibi's Lower Bound Theorem~\protect{\cite[Theorem 1.3]{Hibi}}. This inequality fails for spanning polytopes without interior lattice points, as the prism $[0,1]^2 \times [0,k]$ shows (its $h^*$-vector equals $(1,4k,2k-1)$). The following easy argument implies that non-spanning $3$-polytopes satisfy a stronger inequality, even if some of them do not have interior lattice points:

\begin{proposition}
\label{prop:hstar}
Let $P$ be a lattice $3$-polytope of index $q>1$. Then, $h_2^*(P)\ge (q-1)(1 + h_1^*(P))$.
\end{proposition}

\begin{proof}
Let $\Lambda'\subset \Z^3$ be the lattice spanned by $P\cap \Z^3$ and let $P'$ be the polytope $P$ considered with respect to $\Lambda'$ (equivalently, let $P'=\phi(P)$ where $\phi:\R^3\to \R^3$ is an affine map extending a lattice isomorphism $\Lambda' \overset\cong\to \Z^3$). 

Let $V'$ and $V=qV'$ be the volumes of $P'$ and $P$. 
Observe that the set of lattice points (in particular, the parameters $n$ and $n_0$ in Equation~\eqref{eq:h}) does not depend on whether we look at one or the other lattice.
Then, using the expression for $h^*_2$ in Equation~\eqref{eq:h}
 we get:
\begin{align*}
h_2^*(P) &= h_2^*(P') - V' + V = \\
&=h_2^*(P') + (q-1) V' =\\
&= h_2^*(P') + (q-1) \sum_i h^*_i(P') \ge \\
&\ge (q-1) (1+h^*_1(P')) = (q-1) (1+ h^*_1(P)).
\end{align*}
\end{proof}

Concerning gaps (part (1) of Theorem~\ref{thm:HKN}), empty tetrahedra of volume $q>1$ do have gaps (their $h^*$-vectors are $(1, 0, q-1)$. Clearly, every other non-spanning polytope has $h^*_1 = n-4 > 0$ and, by Proposition~\ref{prop:hstar}, it has $h^*_2 > 0$ too. Thus, empty tetrahedra are the only lattice $3$-polytopes with gaps.

\end{document}